\documentclass{llncs}

\usepackage{mathtools,enumitem}
\usepackage{amssymb,stmaryrd,mathrsfs,dsfont}
\usepackage[ruled,linesnumbered]{algorithm2e}
\usepackage{algpseudocode}

\input xy
\xyoption{all}

\usepackage{tikz}

\tikzstyle{vertex}=[circle, draw, inner sep=0pt, minimum size=3pt]
\newcommand{\vertex}{\node[vertex]}

\usetikzlibrary{decorations.markings}
\begin{document}
\frontmatter          
\pagestyle{headings}  

\mainmatter              

\title{Word-Representability of Well-Partitioned Chordal Graphs}

\titlerunning{Word-Representability of Well-Partitioned Chordal Graphs}  
%

\author{Tithi Dwary \and 
K. V. Krishna}

\authorrunning{Tithi Dwary \and K. V. Krishna} 

\institute{Indian Institute of Technology Guwahati, Guwahati, India,\\
	\email{tithi.dwary@iitg.ac.in};\;
	\email{kvk@iitg.ac.in}}

\maketitle              

\begin{abstract}
In this paper, we study the word-representability of well-partitioned chordal graphs using split decomposition. We show that every component of the minimal split decomposition of a well-partitioned chordal graph is a split graph. Thus we have a characterization for word-representability of well-partitioned chordal graphs. As a consequence, we prove that the recognition of word-representability of well-partitioned chordal graphs can be done in polynomial time. Moreover, we prove that the representation number of a word-representable well-partitioned chordal graph is at most three. Further, we obtain a minimal forbidden induced subgraph characterization of circle graphs restricted to well-partitioned chordal graphs. Accordingly, we determine the class of word-representable well-partitioned chordal graphs having representation number exactly three.
\end{abstract}


\section{Introduction and Preliminaries}

A word over a finite set of letters is a finite sequence which is written by juxtaposing the letters of the sequence. A subword $u$ of a word $w$, denoted by $u \ll w$, is defined as a subsequence of the sequence $w$. For instance, $aabccb \ll acabbccb$. Let $w$ be a word over a set $X$, and $Y \subseteq X$. Then, $w|_Y$ denotes the subword of $w$ that precisely consists of all occurrences of the letters of $Y$. For example, if $w = acabbccb$, then $w|_{\{a, b\}} = aabbb$. For a word $w$, if $w|_{\{a, b\}}$ is of the form $abab \cdots$ or $baba \cdots$, which can be of even or odd length, we say the letters $a$ and $b$ alternate in $w$; otherwise, we say $a$ and $b$ do not alternate in $w$. A $k$-uniform word is a word in which every letter occurs exactly $k$ times.

In this paper, we consider only simple and connected graphs. A graph $G = (V, E)$ is called a word-representable graph, if there exists a word $w$ over $V$ such that for all $a, b \in V$, $\overline{ab} \in E$ if and only if $a$ and $b$ alternate in $w$. Although, the class of word-representable graphs was first introduced in the context of  Perkin semigroups \cite{perkinsemigroup}, this class of graphs received attention of many authors due to its combinatorial properties. The class of word-representable graphs includes several important classes of graphs such as comparability graphs, circle graphs, $3$-colorable graphs and parity graphs. One may refer to the monograph \cite{words&graphs} for a complete introduction to the theory of word-representable graphs.

A word-representable graph $G$ is said to be $k$-word-representable if there is a $k$-uniform word representing it. In \cite{MR2467435}, It was proved that every word-representable graph is $k$-word-representable, for some $k$. The representation number of a word-representable graph $G$, denoted by $\mathcal{R}(G)$, is defined as the smallest number $k$ such that $G$ is $k$-word-representable. A word-representable graph $G$ is said to be permutationally representable if there is a word of the form $p_1p_2 \cdots p_k$ representing $G$, where each $p_i$ is a permutation on the vertices of $G$; in this case $G$ is called a permutationally $k$-representable graph. The permutation-representation number (in short, \textit{prn}) of $G$, denoted by $\mathcal{R}^p(G)$, is the smallest number $k$ such that $G$ is permutationally $k$-representable. It was shown in \cite{perkinsemigroup} that a graph is permutationally representable if and only if it is a comparability graph - a graph which admits a transitive orientation. Further, if $G$ is a comparability graph, then $\mathcal{R}^p(G)$ is precisely the dimension of an induced partially ordered set (in short, poset) of $G$ (cf. \cite{khyodeno2}). It is clear that for a comparability graph $G$, $\mathcal{R}(G) \le \mathcal{R}^p(G)$. 

The class of graphs with representation number at most two is characterized as the class of circle graphs \cite{Hallsorsson_2011} and the class of graphs with \textit{prn} at most two is the class of permutation graphs \cite{Gallaipaper}. In general, the problems of determining the representation number of a word-representable graph, and the \textit{prn} of a comparability graph are computationally hard \cite{Hallsorsson_2011,yannakakis1982complexity}.

We use the following notations in this paper. Let $G = (V, E)$ be a graph. The neighborhood of a vertex $a \in V$ is denoted by $N_G(a)$, and is defined by $N_G(a) = \{b \in V \mid \overline{ab}\in E\}$. For $A \subseteq V$, the neighborhood of $A$, $N_G(A) = \bigcup_{a \in A} N_G(a) \setminus A$. Further, the subgraph of $G$ induced by $A$ is denoted by $G[A]$. For two sets $A, B \subseteq V$, $G[A, B]$ denotes the bipartite graph with the vertex set $A \cup B$, and the edge set $\{\overline{ab} \in E \mid a \in A, b \in B\}$. We say $A$ is complete to $B$ if $A \cap B = \varnothing$ and each vertex in $A$ is adjacent to every vertex in $B$.

We recall the concepts of split decomposition of a connected graph from \cite{Bouchet_1}. A split of a connected graph $G = (V, E)$ is a bipartition $\{V_1, V_2\}$ of $V$ (i.e., $V = V_1 \cup V_2$ and $V_1 \cap V_2 = \emptyset$) satisfying the following: (i) $|V_1| \geq 2$ and $|V_2| \geq 2$, (ii) $N_G(V_1)$ is complete to $N_G(V_2)$. If a graph has no split, then it is said to be a prime graph. 

A split decomposition of a graph $G = (V, E)$ with split $\{V_1, V_2\}$  is represented as a disjoint union of the induced subgraphs $G[V_1]$ and $G[V_2]$ along with an edge $e = \overline{v_1v_2}$, where $v_1$ and $v_2$ are two new vertices such that $v_1$ and $v_2$ are adjacent to each vertices of $N_G(V_2)$  and $N_G(V_1)$, respectively. By deleting the edge $e$, we obtain two components with vertex sets $V_1 \cup \{v_1\}$ and $V_2 \cup \{v_2\}$ called the split components. The two components are then decomposed recursively to obtain a split decomposition of $G$. 

Note that each split component of a graph $G$ is isomorphic to an induced subgraph of $G$ \cite{cicerone1999extension}. A minimal split decomposition of a graph is a split decomposition whose split components can be cliques, stars and prime graphs such that the number of split components is minimized. While there can be multiple split decompositions of a graph, a minimal split decomposition of a graph is unique \cite{Cunningham_2,Cunningham_1}. 

The concept of split decomposition has a large range of applications including NP-hard optimization \cite{graphcoloring,NPcompleteproblems} and the recognition of certain classes of graphs such as distance-hereditary graphs \cite{DHgraph1}, circle graphs \cite{circlegraph3,circlegraph1}, and parity graphs \cite{cicerone1999extension}. Recently, in \cite{Tithi_2024}, word-representability of graphs was studied with respect to the split decomposition. It was proved in \cite{Tithi_2024} that a graph $G$ is word-representable if and only if all the components of split decomposition of $G$ are word-representable. Moreover, the representation number of $G$ is the maximum of the representation numbers of all components of the split decomposition of $G$. As a consequence, it was established that parity graphs are word-representable \cite{Tithi_2024}.

A connected graph $G = (V, E)$ is a well-partitioned chordal graph if there exist a partition $\mathcal{P}$ of the vertex set $V$ into cliques and a tree $\mathcal{T}$ having $\mathcal{P}$ as a vertex set such that for distinct $A, B \in \mathcal{P}$, (i) the edges between $A$ and $B$ in $G$ form a complete bipartite subgraph whose parts are some subsets of $A$ and $B$, if $A$ and $B$ are adajcent in $\mathcal{T}$, and (ii) there are no edges between $X$ and $Y$ in $G$ otherwise. The class of well-partitioned chordal graphs generalizes the class of split graphs, and is a subclass of the class of chordal graphs. Ahn et al. introduced well-partitioned chordal graphs in \cite{WPC_2022} as a tool for narrowing down complexity gaps for problems that are hard on chordal graphs, and easy on split graphs.  Several problems, e.g., tree 3-spanner problem, transversal
of longest paths and cycles, geodetic set problem which are either hard or open on chordal
graphs were proved to be polynomial-time solvable on well-partitioned chordal graphs \cite{WPC_2022}. A detailed information about well-partitioned chordal graphs can be found in Section \ref{wpc_intro}. 

Note that the recognition of word-representability of split graphs can be done in polynomial time \cite{Kitaev_2024}. However, it is open in the case of chordal graphs. So far there is no result available on the word-representability of well-partitioned chordal graphs. It is evident that not all well-partitioned chordal graphs are word-representable as not all split graphs are word-representable. 

In this paper, using split decomposition as a main tool, we study the word-representability of the class of  well-partitioned chordal graphs. We show that every component of the minimal split decomposition of a well-partitioned chordal graph is a split graph. Consequently, we obtain a characterization for word-representability of well-partitioned chordal graphs, as word-representable split graphs were characterized in the literature. Accordingly, we prove that the recognition of word-representability of well-partitioned chordal graphs can be done in polynomial time. Moreover, we show that the representation number of a word-representable well-partitioned chordal graph is at most three. Further, we obtain a minimal forbidden induced subgraph characterization of circle graphs restricted to well-partitioned chordal graphs. Accordingly, we characterize the class of word-representable well-partitioned chordal graphs which have representation number exactly three.

\section{Well-Partitioned Chordal Graphs} \label{wpc_intro}
In this section, we provide the formal definition of a well-partitioned chordal graph and reconcile some relevant results from \cite{WPC_2022}. A connected graph $G = (V, E)$ is said to be a well-partitioned chordal graph if there exist a partition $\mathcal{P}$ of $V$ and a tree $\mathcal{T} = (V', E')$ having  $\mathcal{P}$ as a vertex set such that the following hold.
\begin{enumerate}
	\item Each part $A$ of $\mathcal{P}$ is a clique in $G$.
	\item For each edge $\overline{AB} \in E'$, there exist subsets $A' \subseteq A$ and $B' \subseteq B$ such that the edge set of the bipartite graph $G[A, B]$ is $A' \times B'$.
	\item For each pair of distinct vertices $A$ and $B$ of $V'$ with $\overline{AB} \notin E'$, the edge set of the bipartite graph $G[A, B]$ is empty.
\end{enumerate}

The class of well-partitioned chordal graphs is hereditary, i.e., closed under induced subgraphs. The tree $\mathcal{T}$ is called a partition tree of $G$, and the elements of $\mathcal{P}$ are called its bags. It is known that a connected well-partitioned chordal graph can have more than one partition tree. A bag $B$ in a partition tree $\mathcal{T}$ is called a leaf bag if the degree of $B$ in $\mathcal{T}$ is one; otherwise it is called an internal bag. Let $A, B \in V'$ be two bags that are adjacent in $\mathcal{T}$. Then, the boundary of $A$ with respect to $B$, denoted by $bd(A, B)$, is defined as $\{a \in A \mid N_G(a) \cap B \neq \varnothing\}$. In view of condition 2 of the definition of a well-partitioned chordal graph, we have the following remark.

\begin{remark}\label{boundary}
If two bags $A$ and $B$ are adjacent in $\mathcal{T}$, then	$bd(A, B)$ is complete to $bd(B, A)$.
\end{remark}

A graph $G$ is called a split graph if the vertex set of $G$ can be partitioned into a clique and an independent set. It can be observed that every split graph is a well-partitioned chordal graph. Moreover, we have the following remark.
\begin{remark}\label{splitgraph}
	A connected well-partitioned chordal graph $G$ is a split graph if and only if there exists a partition tree of $G$ such that it is a star with a clique $C$ as its central bag and each leaf bag is a clique of size one. 
\end{remark}

\section{Word-Representability}

\begin{theorem}\label{prime_components}
	Let $G$ be a well-partitioned chordal graph and $H$ be its minimal split decomposition. Then, every component in $H$ is a split graph. 
\end{theorem}

\begin{proof}
	Note that the components in $H$ are cliques, stars, and prime graphs. Since stars and cliques are split graphs, it is sufficient to prove that every prime component of $H$ is a split graph. Let $L = (V, E)$ be a prime component of $H$ such that it is neither a clique nor a star. Then observe that $|V| \ge 4$; otherwise, $L$ is either a star or a clique. Since $L$ is an induced subgraph of $G$, we have $L$ is also a well-partitioned chordal graph \cite{WPC_2022}. Thus, there exist a partition $\mathcal{P}$ of the vertex set $V$ and a partition tree $\mathcal{T} = (V', E')$ having $V' = \mathcal{P}$ such that all the three conditions given in the definition of a well-partitioned chordal graph are satisfied. 
	
	We now claim that not all bags of $\mathcal{P}$ are of size one. On the contrary, suppose that each bag of the partition $\mathcal{P}$ is of size one. Then, it is evident that $L$ is a tree. Since $L$ is not a star, $L$ can be further decomposed into stars, a contradiction that $L$ is a prime graph. Thus, there exists a bag $B$ of $\mathcal{P}$ which is of size at least two. 
	
	We further claim that the bag $B$ cannot be adjacent to a bag of size strictly bigger than one in the partition tree $\mathcal{T}$. On the contrary, suppose that the bag $B$ is adjacent to a bag $B'$ of size at least two in $\mathcal{T}$ so that $\overline{BB'} \in E'$. Let $\mathcal{T}^B$ and  $\mathcal{T}^{B'}$ be the components of $\mathcal{T} \setminus \overline{BB'}$ (the graph obtained by deleting $\overline{BB'}$ from $\mathcal{T}$) containing $B$ and $B'$, respectively.  Further, let $V^{B}$ and $V^{B'}$ be the vertex sets obtained by taking union of all bags appeared in $\mathcal{T}^B$ and  $\mathcal{T}^{B'}$, respectively. Then, note that $|V^{B}| \ge 2$ (as $B \subseteq V^B$) and $|V^{B'}| \ge 2$ (as $B' \subseteq V^{B'}$). Further, we have $V^{B} \cap V^{B'} = \varnothing$ and  $V^{B} \cup V^{B'} = V$. Moreover, $N_L(V^B) = bd(B', B)$ and $N_L(V^{B'}) = bd(B, B')$. In view of Remark \ref{boundary}, we have $N_L(V^B)$ is complete to $N_L(V^{B'})$ so that $\{V^{B}, V^{B'}\}$ forms a split in $L$, a contradiction that $L$ is a prime graph. Thus, the bag $B$ is adjacent to only size-one bags in $\mathcal{T}$. 
	
	We now claim that each size-one bag that is adjacent to $B$ in $\mathcal{T}$ is a leaf bag. 
	Suppose there is a bag of size-one, say $B'$, in $\mathcal{T}$ such that it is adjacent to $B$ but not a leaf in $\mathcal{T}$. Then, there is another bag, say $B''$, in $\mathcal{T}$ such that $\overline{B'B''} \in E'$. Define the subsets $V^{B}$ and $V^{B'}$ of $V$ similarly as above. Then, note that $|V^{B}| \ge 2$ (as $B \subseteq V^B$) and $|V^{B'}| \ge 2$ (as $B' \cup B'' \subseteq V^{B'}$). Further, we have $V^{B} \cap V^{B'} = \varnothing$ and  $V^{B} \cup V^{B'} = V$. Moreover, we have $N_L(V^B) = bd(B', B)$ and $N_L(V^{B'}) = bd(B, B')$. In view of Remark \ref{boundary}, we have $N_L(V^B)$ is complete to $N_L(V^{B'})$ so that $\{V^{B}, V^{B'}\}$ forms a split in $L$, a contradiction that $L$ is a prime graph. Thus, the partition tree $\mathcal{T}$ of $L$ is a star with the bag $B$ as its central clique and each leaf bag is a clique of size-one. Hence, in view of Remark \ref{splitgraph}, we have the prime graph $L$ is a split graph. \qed 
\end{proof}

\begin{remark}
	Note that the converse of Theorem \ref{prime_components} is not true. For instance, the split components of $C_4$, a cycle of length four, are stars. However, $C_4$ is not a chordal graph, and hence not a well-partitioned chordal graph.
\end{remark}

From \cite{Kitaev_2021,Kitaev_2024}, we now recall the characterization of word-representable split graphs as per the following result. Note that for any two integers $a \le b$, the set of integers $\{a, a+1, \ldots, b\}$ is denoted by $[a, b]$.

\begin{theorem}[\cite{Kitaev_2021,Kitaev_2024}]\label{Word_split_graph}
	Let $G = (I \cup C, E)$ be a split graph such that $I$ and $C$ induce an independent set and a clique, respectively, in $G$. Then, $G$ is word-representable if and only if the vertices of $C$ can be labeled from $1$ to $k = |C|$ in such a way that for each $a, b \in I$ the following holds.
	\begin{enumerate}[label=\rm (\roman*)]
		\item \label{point_1} Either $N(a) = [1, m] \cup [n, k]$, for $m < n$, or $N(a) = [l, r]$, for $l \le r$.
		\item \label{point_2} If $N(a) = [1, m] \cup [n, k]$ and $N(b) = [l, r]$, for $m < n$ and $l \le r$, then $l > m$ or $r < n$.
		\item \label{point_3} If $N(a) = [1, m] \cup [n, k]$ and $N(b) = [1, m'] \cup [n', k]$, for $m < n$ and $m' < n'$, then $m' < n$ and $m < n'$.
	\end{enumerate}
\end{theorem}

Thus, in view of theorems \ref{prime_components}, \ref{Word_split_graph} and \cite[Theorem 3.4]{Tithi_2024}, we characterize word-representable well-partitioned chordal graphs as per the following.

\begin{corollary}\label{char_wpc}
	A well-partitioned chordal graph $G$ is word-representable if and only if all the prime components of the minimal split decomposition of $G$ are word-representable split graphs.
\end{corollary}

\begin{theorem}\label{wpc_rp_no}
	If $G$ is a word-representable well-partitioned chordal graph, then $\mathcal{R}(G) \le 3$.
\end{theorem}

\begin{proof}
	Let $H_i$ $(1 \le i \le k)$ be the components of the minimal split decomposition of $G$. In view of \cite[Theorem 3.4]{Tithi_2024}, since $G$ is word-representable, we have each of $H_i$ is word-representable and $\mathcal{R}(G) = \max_{1 \le i \le k} \mathcal{R}(H_i)$. Further, since each $H_i$ is a word-representable split graph (by Theorem \ref{prime_components}), we have $\mathcal{R}(H_i) \le 3$ (by \cite[Theorem 5]{Tithi_2025}). Hence, we have  $\mathcal{R}(G) \le 3$. \qed
\end{proof}

\begin{theorem}\label{wpc_circle_graph}
	Let $G$ be a well-partitioned chordal graph. Then, $G$ is a circle graph if and only if $G$ is a $\mathcal{C}$-free graph, where $\mathcal{C}$ is the class of graphs given in \cite[Fig. 2]{Tithi_2025}.
\end{theorem}

\begin{proof}
	Since each graph belonging to the class $\mathcal{C}$ is not a circle graph, if $G$ is a circle graph, then $G$ is $\mathcal{C}$-free. Conversely, suppose that $G$ is a $\mathcal{C}$-free graph. We prove that $G$ is a circle graph. On the contrary, suppose that $G$ is not a circle graph. Let $H_i$ $(1 \le i \le k)$ be the components of the minimal split decomposition of $G$. Then, there exists at least one component, say $H_t$, such that $H_t$ is not a circle graph; otherwise, if for each $1 \le i \le k$, $H_i$ is a circle graph, i.e., a $2$-word-representable graph, by \cite[Theorem 3.4]{Tithi_2024}, we have $G$ is a $2$-word-representable graph (hence, a circle graph), a contradiction to $G$ is not a circle graph. In view of Theorem \ref{prime_components}, since $H_t$ is a split graph, by \cite[Theorem 1.1]{Split_circle_graphs}, $H_t$ must contain at least one graph from the class $\mathcal{C}$ as an induced subgraph. Since $H_t$ is an induced subgraph of the graph $G$, $G$ contains at least one graph from the class $\mathcal{C}$ as an induced subgraph, a contradiction to $G$ is $\mathcal{C}$-free. Hence, $G$ is a circle graph. \qed
\end{proof}

Since each graph belonging to the class $\mathcal{C}$ is a minimally non-circle\footnote{A graph $G$ is minimally non-circle if $G$ is not a circle graph but every proper induced subgraph of $G$ is a circle graph.} graph \cite[Theorem 3.44]{Split_circle_graphs}, Theorem \ref{wpc_circle_graph} provides a minimal forbidden induced subgraph characterization of circle graphs restricted to well-partitioned chordal graphs. Further, we have the following proposition.

\begin{proposition}
	Each graph belonging to the class $\mathcal{C}$ is a prime graph.
\end{proposition}

\begin{proof}
	 For $G \in \mathcal{C}$, if $G$ is not a prime graph, then $G$ can be decomposed using split decomposition. Let $H_i$, $1 \le i \le k$, be the components of a split decomposition of $G$. Note that each $H_i$ is a proper induced subgraph of $G$. Thus, each $H_i$ is a circle graph as $G$ is a minimally non-circle graph \cite[Theorem 3.44]{Split_circle_graphs}. Then, in view of \cite[Theorem 3.4]{Tithi_2024}, we have $G$ is a circle graph, a contradiction. \qed
\end{proof}

\begin{figure}[t]
	\centering
	\begin{minipage}[b]{.25\textwidth}
		\centering
		\[\begin{tikzpicture}[scale=0.6]

			\vertex (1) at (0,0) [fill=black,label=left:$ $] {};
			\vertex (2) at (1.5,0) [fill=black,label=left:$ $] {};
			\vertex (3) at (0,1) [fill=black,label=left:$ $] {};	
			\vertex (4) at (1.5,1) [fill=black,label=left:$ $] {};
			\vertex (5) at (0.75,2) [fill=black,label=left:$ $] {};	
			\vertex (6) at (0.75,3) [fill=black,label=left:$ $] {};

			\path 
			(1) edge (3)
			(2) edge (4)
			(3) edge (5)
			(4) edge (5)
			(5) edge (6)
			(3) edge (4);

		\end{tikzpicture}\] 
		$W_1$	
	\end{minipage}%
	\begin{minipage}[b]{.25\textwidth}
		\centering
		
		\[\begin{tikzpicture}[scale=0.6]

			\vertex (1) at (0,0) [fill=black,label=left:$ $] {};
			\vertex (2) at (1.5,0) [fill=black,label=left:$ $] {};
			\vertex (3) at (3,0) [fill=black,label=left:$ $] {};	
			\vertex (4) at (0.75,1) [fill=black,label=left:$ $] {};
			\vertex (5) at (2.25,1) [fill=black,label=left:$ $] {};	
			\vertex (6) at (1.5,2) [fill=black,label=left:$ $] {};

			\path
			(1) edge (2)
			(2) edge (3)
			(1) edge (4)
			(2) edge (4)
			(2) edge (5)
			(3) edge (5)
			(4) edge (5)
			(4) edge (6)
			(5) edge (6); 
			
		\end{tikzpicture}\]
		$W_2$
	\end{minipage}%
	\begin{minipage}[b]{.25\textwidth}
		\centering
		
		\[\begin{tikzpicture}[scale=0.6]
			
			\vertex (1) at (0,0) [fill=black,label=left:$ $] {};
			\vertex (2) at (1.5,0) [fill=black,label=left:$ $] {};
			\vertex (3) at (0,1) [fill=black,label=left:$ $] {};	
			\vertex (4) at (1.5,1) [fill=black,label=left:$ $] {};
			\vertex (5) at (0.75,2) [fill=black,label=left:$ $] {};	
			\vertex (6) at (-0.75,2) [fill=black,label=left:$ $] {};	
			\vertex (7) at (2.25,2) [fill=black,label=left:$ $] {};

			\path 
			(1) edge (3)
			(2) edge (4)
			(3) edge (5)
			(4) edge (5)
			(5) edge (6)
			(3) edge (4)
			(3) edge (6)
			(5) edge (7)
			(4) edge (7);

		\end{tikzpicture}\]
		$W_3$
	\end{minipage}%
	\begin{minipage}[b]{.25\textwidth}
		\centering
		
		\[\begin{tikzpicture}[scale=0.35]
			
			\vertex (1) at (0,0) [fill=black,label=below:$ $] {};
			\vertex (2) at (-1.4,1.8) [fill=black,label=below:$ $] {};
			\vertex (3) at (1.4,1.8) [fill=black,label=right:$ $] {};
			\vertex (4) at (-2.8,0) [fill=black,label=left:$ $] {};
			\vertex (5) at (2.8,0) [fill=black,label=above:$ $] {};
			\vertex (6) at (-2.8,-1.8) [fill=black,label=right:$ $] {};
			\vertex (7) at (2.8,-1.8) [fill=black,label=left:$ $] {};
			
			\path 
			(1) edge (2)
			(1) edge (3)
			(1) edge (4)
			(1) edge (5)
			(2) edge (3)
			(2) edge (4)
			(3) edge (5)
			(4) edge (6)
			(5) edge (7);
			
		\end{tikzpicture}\]
		$W_4$
	\end{minipage}%
	\caption{The family of graphs $\mathcal{W}$}
	\label{fig1}	
\end{figure}
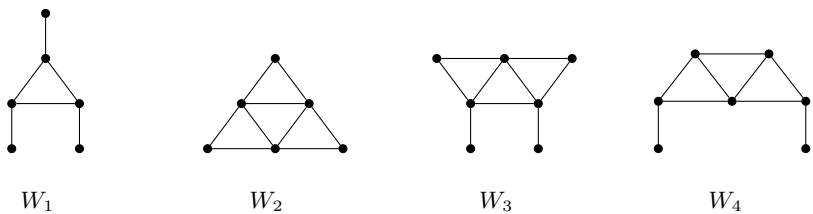

\begin{corollary}
	Let $G$ be a word-representable well-partitioned chordal graph. Then, $\mathcal{R}(G) = 3$ if and only if $G$ contains at least one of the following graphs as an induced subgraph: even-$k$-sun (for every $k \ge 4$), $F_0$, $F_1(k), F_2(K)$, for odd $k \ge 5$ (depicted in \cite[Fig. 2]{Tithi_2025}).
\end{corollary}

\begin{proof}
	 Suppose that $\mathcal{R}(G) = 3$ so that $G$ is not a circle graph. Then, from Theorem \ref{wpc_circle_graph}, $G$ contains at least one graph from the family $\mathcal{C}$ as an induced subgraph. Further, in view of \cite[Lemma 4]{Tithi_2025}, as $G$ is a word-representable graph, we see that $G$ contains at least one of the following graphs as an induced subgraph: even-$k$-sun (for even $k \ge 4$), $F_0$, $F_1(k)$, $F_2(k)$, for odd $k \ge 5$.
	
	Conversely, suppose that $G$ contains at least one of the following graphs as an induced subgraph: even-$k$-sun (for even $k \ge 4$), $F_0$, $F_1(k)$, $F_2(k)$, for odd $k \ge 5$. Since each of these graphs belongs to the family $\mathcal{C}$, by Theorem \ref{wpc_circle_graph}, $G$ is not a circle graph so that $\mathcal{R}(G) > 2$. Further, since $\mathcal{R}(G) \le 3$ (by Theorem \ref{wpc_rp_no}), we have $\mathcal{R}(G) = 3$. \qed
\end{proof}

From the forbidden induced subgraph characterizations of both well-partitioned chordal graphs \cite{WPC_2022} and comparability graphs \cite{Gallaipaper}, the following result can be ascertained.

\begin{theorem}
	Let $G$ be a well-partitioned chordal graph. Then, $G$ is a comparability graph if and only if $G$ is $\mathcal{W}$-free, where $\mathcal{W}$ is the class of graphs given in Fig. \ref{fig1}.
\end{theorem}

A poset is said to be cycle-free  if the corresponding comparability graph is a chordal graph. It was proved in \cite[Theorem 1]{cyclefreeposet_1991} that every cycle-free poset has dimension at most four. Accordingly, we have the following result on the \textit{prn} of a well-partitioned chordal comparability graph.

\begin{theorem}
	Let $G$ be a well-partitioned chordal graph. If $G$ is a comparability graph, then $\mathcal{R}^p(G) \le 4$.
\end{theorem}

\subsection{Recognition algorithm}

Recall that not all well-partitioned chordal graphs are word-representable, since not all graphs in the subclass of split graphs are word-representable. In this section, we focus on the following recognition problem.

\subsubsection{Problem}
	Given a well-partitioned chordal graph, is it word-representable?

\subsubsection{Algorithm}
    In view of the characterization given in Corollary \ref{char_wpc}, it can be observed that the above problem can be solved by decomposing $G$ and verifying whether all the components are word-representable or not. Algorithm \ref{algo-1} performs this test.

\begin{algorithm}[h]\label{Algo_1}
	\caption{Recognizing word-representability of well-partitioned chordal graphs.}
	\label{algo-1}
	\KwIn{A well-partitioned chordal graph $G$.}
	\KwOut{Yes if $G$ is word-representable; otherwise, No.}
	
	Compute the minimal split decomposition $H$ of $G$. Let $H_i$, $1 \le i \le k$, be the components of $H$.\\
	Flag = Yes.\\
	\For{each $H_i \in H$}{\If{$H_i$ is not a word-representable graph}{Flag = No.}}
	\Return{{\rm Flag}}
\end{algorithm}

\subsubsection{Complexity}
The minimal split decomposition can be computed in linear time \cite{pierre_lt_algo,dahlhaus1994} (step 1). Note that the number of components of the decomposition are polynomially bounded with respect to the size of the input graph \cite{Stefano_2012}. Since every $H_i$ is a split graph (by Theorem \ref{prime_components}), the word-representability of each $H_i$ can be verified in polynomial time (step 4) on the size of $H_i$ \cite{Kitaev_2024}. Thus, testing all the components (step 3) takes polynomial time on the size of the input graph. Hence, we have the following theorem.

\begin{theorem}
	The word-representability of a well-partitioned chordal graph can be decided in polynomial time.
\end{theorem}

\end{document}